\documentclass[a4paper,10pt]{article}

\usepackage{amsmath,amssymb,amsthm,graphicx,
	subfig,
verbatim,cite,
} 

\usepackage{enumitem}
\usepackage{mathtools}

\usepackage{geometry}
\usepackage{tikz}
\usetikzlibrary{shapes.geometric,calc,decorations.markings,math}

\tikzstyle{nodo}=[circle,draw,fill,inner sep=0pt,minimum size=%
1.5mm]
\tikzstyle{infinito}=[circle,inner sep=0pt,minimum size=0mm]

\newcommand\R{{\mathbb R}}

\newcommand\Hmu{{\mathcal{M}_\mu}}

\newcommand\Sf{\mathcal S}

\newcommand\EE{{\mathcal E}}

\newcommand\f{\frac}

\newcommand\ee{{\mathcal E}}
\newcommand\JJ{{\mathcal J}}
\newcommand\JJn{{\mathcal J_\Omega^{nod}}}

\newcommand\MM{{\mathcal M}}

\newcommand{\intervaloo}[1]{\mathopen(#1\mathclose)}

\newcommand\eps{\varepsilon}

\newcommand\C{\mathcal C}

\newcommand\NN{\mathcal N}

\newcommand\barlam{\overline{\lambda}}
\newcommand\bargi{\overline{g}}

\newcommand\NL{{\mathcal N}_\lambda}

\newcommand{\wu}{{\widetilde u}}
\newcommand{\wlambda}{{\widetilde \lambda}}

\theoremstyle{definition}

\theoremstyle{plain}
\newtheorem{theorem}{Theorem}[section]
\newtheorem{proposition}[theorem]{Proposition}
\newtheorem{lemma}[theorem]{Lemma}

\newcounter{ass}
\theoremstyle{remark}
\newtheorem{remark}[theorem]{Remark}
\newtheorem*{remark*}{Remark}

\theoremstyle{definition}

\DeclareMathOperator{\supp}{supp}

\numberwithin{equation}{section}

\date{}

\title{An action approach to nodal and least energy normalized solutions for nonlinear Schr\"odinger equations}

\author{Colette De Coster$^1$, Simone Dovetta$^2$, Damien Galant$^{1,3}$, Enrico Serra$^2$ 
	\\ \ \\{\small$^1$ Univ. Polytechnique Hauts-de-France, INSA Hauts-de-France, CERAMATHS - Laboratoire de}
	\\ {\small Mat\'eriaux C\'eramiques et de Math\'ematiques, F-59313 Valenciennes, France}
	 \\{\small$^2$Dipartimento di Scienze
		Matematiche ``G.L. Lagrange'', Politecnico di Torino } \\ {\small
		Corso Duca degli Abruzzi, 24, 10129 Torino, Italy} \\ 
		{\small$^3$ {\small  F.R.S.-FNRS and UMONS - Universit\'e de Mons, Mons, Belgium}
		}}

\begin{document}

\maketitle

\begin{abstract}
 We develop a new approach to the investigation of normalized solutions for nonlinear Schr\"odinger equations based on the analysis of the masses of ground states of the corresponding action functional. Our first result is a complete characterization of the masses of action ground states, obtained via a Darboux-type property for the derivative of the action ground state level. We then exploit this result to tackle normalized solutions with a twofold perspective. First, we prove existence of normalized nodal solutions for every mass in the $L^2$-subcritical regime, and for a whole interval of masses in the $L^2$-critical and supercritical cases. Then, we show when least energy normalized solutions/least energy normalized nodal solutions are  action ground states/nodal action ground states.     
\end{abstract}

\smallskip


\section{Introduction}

The present paper focuses on {\em normalized} solutions of nonlinear Schr\"odinger (NLS) equations with homogeneous Dirichlet boundary conditions on bounded domains, namely solutions of the problem
\begin{equation}
\label{nlse}
\begin{cases}
-\Delta u +\lambda u = |u|^{p-2}u  & \text{in }\Omega\\
u=0 & \text{on }\partial\Omega\\
\|u\|_{L^2(\Omega)}^2=\mu.
\end{cases}
\end{equation}
Here, $\Omega\subset\R^N$ is a connected bounded open set,  $\lambda$ and $\mu$ are real parameters, and the nonlinearity exponent satisfies
\[
p\in(2,2^*),\quad 2^*=\frac{2N}{N-2}\quad\left(2^*=\infty\text{ if }N=1,2\right).
\]
The attribute normalized for a function $u$ solving \eqref{nlse} comes from the fact that its $L^2$-norm (usually called the {\em mass}) is prescribed a priori. In our setting, this means that the parameter $\mu>0$ is given, whereas $\lambda$ (sometimes called the chemical potential or the frequency) is an unknown of the problem. As is well known, the problem has a variational structure, since weak solutions of \eqref{nlse} are critical points of the {\em energy} functional $E:H_0^1(\Omega)\to\R$
\begin{equation*}
E\left(u,\Omega\right):=\frac12\|\nabla u\|_{L^2(\Omega)}^2-\frac1p\|u\|_{L^p(\Omega)}^p
\end{equation*}
on the $L^2$-sphere
\[
\Hmu(\Omega):=\left\{u\in H_0^1(\Omega)\,:\,\|u\|_{L^2(\Omega)}^2=\mu\right\},
\]
$\lambda$ arising then as a Lagrange multiplier.

The study of normalized solutions for NLS equations gathered a constantly growing interest in the last decades. In particular, among all solutions with fixed mass, a specific attention can be naturally devoted to {\em least energy normalized solutions}, i.e. functions $u$ solving \eqref{nlse} and satisfying
\[
E(u,\Omega)=\inf\left\{E(v,\Omega)\,:\, v\text{ solves }\eqref{nlse}\text{ for some }\lambda\in\R\right\}.
\]
In the seminal papers \cite{CL82,J97} least energy positive solutions are identified for the problem on the whole $\R^N$ in the $L^2$-subcritical $p<2+\frac4N$ and $L^2$-supercritical $p>2+\frac4N$ regimes, respectively. When $p<2+\frac4N$, these least energy solutions can be found by solving the minimization problem
\[
\inf_{u\in\Hmu(\R^N)}E(u,\R^N),
\]
which is attained for every $\mu>0$. Conversely, when $p$ is $L^2$-supercritical this is no longer possible, as the energy $E$ is unbounded from below on $\Hmu(\R^N)$ for every $\mu$, and different approaches (e.g. of mountain pass type) are needed. Since \cite{J97}, normalized solutions on $\R^N$ have been largely investigated in various settings (see e.g. \cite{BdV,BMRV,BS,BZZ, JL, JZZ, JZZ2, MRV, S1, S2,WW} and references therein) and a comprehensive theory is by now available. On the contrary, on bounded domains the literature is more limited, to date, and the general portrait is less understood. Given the boundedness of the domain $\Omega$, in the $L^2$-subcritical regime $p<2+\frac4N$ least energy solutions always exist and they are again the global minimizers of $E$ on the whole set $\Hmu(\Omega)$. The same is true for masses smaller than a threshold (independent of $\Omega$) in the $L^2$-critical case $p=2+\frac4N$. When $p>2+\frac4N$, instead, even existence of positive solutions (not necessarily least energy) is more involved, because many crucial properties of the problem on $\R^N$ are no longer available on bounded domains (as e.g. the invariance under dilations of the ambient space). To the best of our knowledge, a complete description of the set of normalized positive solutions is available only when $\Omega$ is a ball and it is given in \cite{NTV1} (similar results have then been obtained also for NLS systems in \cite{NTV2}). As for general domains, existence results for solutions (not necessarily positive) with fixed Morse index have been derived when $p$ is $L^2$-supercritical in \cite{PV,PVY}, whereas in \cite{PPVV} specific positive solutions are constructed for large masses when $p<2+\frac4N$, small masses when $p>2+\frac4N$, and masses close to an explicit value when $p=2+\frac4N$. 

\smallskip
The aim of the present paper is to fit in this research line focusing on the following questions:
\begin{enumerate}
	\item[1)] how to find least energy normalized solutions in the $L^2$-supercritical regime?
	\item[2)] how to find normalized nodal solutions?
\end{enumerate} 
As far as we know, to date both questions are essentially open. As for 1), the only available result we are aware of is the already mentioned one on the ball reported in \cite{NTV1}, that identifies the normalized solution with minimal energy among {\em positive} ones. Even in this special setting, it is not known whether this is also the least energy solution among {\em all} solutions with the same mass. For domains other than the ball, nothing seems to be known. The situation concerning 2) is even worse, due to the lack of general existence results for normalized nodal solutions. Actually, all papers in the literature either restrict their attention to positive solutions, or do not allow to recover any specific information on the sign of the solutions under exam. 

Though at a first sight questions 1) and 2) may appear somehow far from each other, a common feature that may perhaps explain the lack of results in both directions is the absence of suitable variational frameworks to tackle them. 

This is readily understood when looking for $L^2$-supercritical least energy solutions, for which we already observed that it is not possible to simply minimize $E$ on the whole manifold $\Hmu(\Omega)$ (as one does in the $L^2$-subcritical regime). 

Such a difficulty is all the  more severe for normalized nodal solutions, for which a proper variational framework involving the energy is not available even when $p$ is $L^2$-subcritical. For instance, one may be first tempted to consider
\[
\inf_{\substack{u\in\Hmu(\Omega) \\u^\pm\not\equiv0}}E(u,\Omega),
\] 
where $u^+$ and $u^- = \min (u,0)$ are the positive and negative parts of $u$, but it is evident that this number is never attained, as it  coincides with the infimum of $E$ on the whole $\Hmu(\Omega)$ (with no sign constraint). 
Even a slightly more sophisticated approach considering the two-parameter minimization problem 
\[
\inf_{\substack{u^+\in \MM_{\mu_1}(\Omega), u^-\in \MM_{\mu_2}(\Omega) \\ \mu=\mu_1+\mu_2}}E(u,\Omega),
\]
leads to seemingly insuperable difficulties.

\smallskip
In this paper, we tackle both 1) and 2) using a unified approach: we take advantage of already available existence results for solutions of the problem
\begin{equation}
\label{nlseNOMASS}
\begin{cases}
-\Delta u +\lambda u = |u|^{p-2}u  & \text{in }\Omega\\
u=0 & \text{on }\partial\Omega
\end{cases}
\end{equation}
for a {\em given} $\lambda\in\R$, that is \eqref{nlse} without the mass constraint, and  we characterize the dependence on $\lambda$ of the mass of such solutions. Indeed, for a fixed $\lambda\in\R$, it is well known that positive (up to a change of sign) solutions of \eqref{nlseNOMASS} can be  found variationally e.g. by considering (for any $p \in (2,2^*)$) the minimization problem  
\begin{equation*}
\JJ_\Omega(\lambda):=\inf_{u\in\NN_\lambda(\Omega)}J_\lambda(u,\Omega)
\end{equation*}
for the {\em action} functional $J_\lambda(\,\cdot\,, \Omega) :H_0^1(\Omega)\to\R$
\begin{equation*}
J_\lambda(u,\Omega):=\frac12\|\nabla u\|_{L^2(\Omega)}^2+\frac\lambda2\|u\|_{L^2(\Omega)}^2-\frac1p\|u\|_{L^p(\Omega)}^p
\end{equation*}
constrained to the associated {\em Nehari manifold}
\[
\begin{split}
\NN_\lambda(\Omega):=&\,\left\{u\in H_0^1(\Omega)\setminus\{0\} \,:\, J_\lambda'(u,\Omega)u=0\right\}\\
=&\,\left\{u\in H_0^1(\Omega)\setminus\{0\}  \,:\,\|\nabla u\|_{L^2(\Omega)}^2+\lambda\|u\|_{L^2(\Omega)}^2=\|u\|_{L^p(\Omega)}^p\right\}.
\end{split}
\]
Similarly, when looking for nodal solutions one can consider the problem
\begin{equation*}
\JJ_\Omega^{nod}(\lambda):=\inf_{\NN_\lambda^{nod}(\Omega)}J_\lambda(u,\Omega),
\end{equation*}
where
\[
\NN_\lambda^{nod}(\Omega):=\left\{u\in H_0^1(\Omega)\,:\,u^\pm\in\NN_\lambda(\Omega)\right\}
\]
is the {\em nodal Nehari set}. Depending on the value of $\lambda$, existence of solutions of these two problems, usually called {\em action ground states} and {\em nodal action ground states} respectively, is essentially well known (see Section \ref{sec:prel} below for further details). 

The main contribution of the present paper is the following complete characterization of the masses of all action and nodal action ground states. 
\begin{theorem}
	\label{thm:masses}
	Let $\Omega\subset \R^N$ be open and bounded and, for every $p\in(2,2^*)$, let
	\begin{equation}
	\label{M}
	\begin{split}
	M_p(\Omega):=&\,\left\{\|u\|_{L^2(\Omega)}^2\,:\, u\in\NN_\lambda(\Omega)\text{ and }J_{\lambda}(u,\Omega)=\JJ_\Omega(\lambda)\text{ for some }\lambda\in\R\right\}\\
	M_p^{nod}(\Omega):=&\,\left\{\|u\|_{L^2(\Omega)}^2\,:\, u\in\NN_\lambda^{nod}(\Omega)\text{ and }J_{\lambda}(u,\Omega)=\JJ_\Omega^{nod}(\lambda)\text{ for some }\lambda\in\R\right\}
	\end{split}
	\end{equation}
	be the set of masses of all action ground states and nodal action ground states, respectively. Then
	\begin{itemize}
		\item[(i)] if $p<2+\frac4N$, then $M_p(\Omega)=M_p^{nod}(\Omega)=(0,\infty)$;
		
		\item[(ii)] if $p=2+\frac4N$, then there exist $0<\mu_p, \mu_p^{nod}<\infty$ such that either $M_p=(0,\mu_p)$ or $M_p=(0,\mu_p]$, and either $M_p^{nod}=(0,\mu_p^{nod})$ or $M_p^{nod}=(0,\mu_p^{nod}]$;
		
		\item[(iii)] if $p>2+\frac4N$, then there exist $0<\mu_p, \mu_p^{nod}<\infty$ such that $M_p=(0,\mu_p]$ and $M_p^{nod}=(0,\mu_p^{nod}]$.
	\end{itemize}
\end{theorem}
Notice that Theorem \ref{thm:masses} holds for any bounded open subset $\Omega$ in $\R^N$, and this high generality makes its proof far from trivial. Indeed, taking for granted the existence of action ground states (resp. nodal action ground states) $u_\lambda$ at fixed $\lambda$, one may be tempted to try to characterize the set $M_p$ (resp. $M_p^{nod}$) by studying the map $\lambda\mapsto \|u_\lambda\|_{L^2(\Omega)}$. However, in principle such a map is not even well defined, as ground states need not be unique, and in any case its regularity is by no means guaranteed. Actually, in other contexts where the dependence on $\lambda$ of the mass of a curve of solutions $u_\lambda$ is relevant, it is common to assume from the very beginning to work with a $C^1$ curve of solutions (as one does e.g. in the standard stability theory for Hamiltonian systems \cite{GSS,SS,W}). For ground states on a general bounded domain, this level of regularity is too strong. Instead, the proof of Theorem \ref{thm:masses} does not require any regularity assumption of this sort and exploits a different perspective. Roughly, we will show that $M_p$ (resp. $M_p^{nod}$) is the range of the derivative of the action ground state level $\JJ_\Omega$ (resp. of $\JJ_\Omega^{nod}$) so that Theorem \ref{thm:masses} can also be seen as a Darboux-type theorem for $\JJ_\Omega$ and $\JJ_\Omega^{nod}$ (see Remark \ref{darboux}). 

Let us now discuss the impact of Theorem \ref{thm:masses} on questions 1)-2) above. With respect to 2), since even the existence of one nodal solution of \eqref{nlse} at prescribed mass $\mu$ is in general an open problem, we have the following result, that is an immediate corollary of Theorem \ref{thm:masses}.

\begin{theorem}
	\label{thm:exnod}
	Let $\Omega\subset \R^N$ be open and bounded, and $p\in(2,2^*)$. Then
	\begin{itemize}
		\item[(i)] if $p<2+\frac4N$, there exists a nodal solution of \eqref{nlse} for every $\mu>0$;
		\item[(ii)] if $p=2+\frac4N$, there exists a nodal solution of \eqref{nlse} for every $\mu\in(0,\mu_p^{nod})$, where $\mu_p^{nod}$ is as in Theorem \ref{thm:masses}(ii);
		\item[(iii)] if $p>2+\frac4N$, there exists a nodal solution of \eqref{nlse} for every $\mu\in(0,\mu_p^{nod}]$, where $\mu_p^{nod}$ is as in Theorem \ref{thm:masses}(iii).
	\end{itemize}
\end{theorem}

\begin{remark}
	Clearly, a statement analogous to Theorem \ref{thm:exnod} can be given for normalized positive solutions of \eqref{nlse} too, with $\mu_p$ in place of $\mu_p^{nod}$. When $p\leq 2+\frac4N$, this does not extend the existence results for positive solutions already available in the literature. On the contrary, it is more relevant in the $L^2$-supercritical case. Indeed, in this regime, our approach provides a simple technique to exhibit normalized solutions for a whole interval of masses $(0,\mu_p]$ and, as far as we know, similar results have been previously obtained only through more technically demanding constructions (see e.g. \cite{PV}).
\end{remark}

Since Theorem \ref{thm:exnod} provides regimes of nonlinearities and masses for which the set of normalized nodal solutions is not empty, it is then natural to wonder whether one can identify the {\em least energy nodal solutions}, i.e. $u$ solving \eqref{nlse} such that $u^\pm\neq0$ and
\[
E(u,\Omega)=\inf\{E(v,\Omega)\,:\, v\text{ is a nodal solution of }\eqref{nlse}\text{ for some }\lambda\in\R\}.
\]
Actually, our method allows to answer in the affirmative to this question in the $L^2$-subcritical and $L^2$-critical regimes. To state this result, let 
\begin{equation}
\label{eq:muN}
\mu_N:= 2\inf_{u \in \NN_1(\R^N)} \left(\frac12 \|\nabla u\|_{L^2(\R^N)}^2 + \frac12 \| u\|_{L^2(\R^N)}^2 - \frac{1}{2+ 4/N}\|u\|_{L^{2+4/N}(\R^N)}^{2+4/N}\right).
\end{equation}
\begin{theorem}
	\label{thm:LENS}
	Let $\Omega\subset \R^N$ be open and bounded, and either 
	\begin{itemize}
		\item[(i)] $p<2+\frac4N$ and $\mu>0$; or
		\item[(ii)] $p=2+\frac4N$ and $\mu<2\mu_N$, where $\mu_N$ is the number in \eqref{eq:muN}.
	\end{itemize}
	Then there exists a least energy normalized nodal solution with mass $\mu$. Moreover, every least energy normalized nodal solution $u$ is a nodal action ground state in $\NN_\lambda^{nod}(\Omega)$, where $\lambda$ is the number associated to $u$ in \eqref{nlse}.
\end{theorem}
Theorem \ref{thm:LENS} says that, in the above regimes, least energy normalized nodal solutions are nodal action ground states. Observe that, at the critical power $p=2+\frac4N$, we are able to prove this fact only for masses strictly smaller than the threshold $2\mu_N \le \mu_p^{nod}$, although nodal solutions exist for every $\mu \le \mu_p^{nod}$ (Theorem \ref{thm:masses}). 

\begin{remark}
	Again, the analogue of Theorem \ref{thm:LENS} can be stated and proved for least energy normalized positive solutions. In this case, solutions with minimal energy exist for every mass when $p<2+\frac4N$, and for every mass strictly smaller than $\mu_N$ when $p=2+\frac4N$, and they are also action ground states in $\NN_\lambda(\Omega)$ for suitable $\lambda$. However, this is already well known, since in this range of $p$ and $\mu$ it is easily seen that $E$ admits global minimizers on $H_\mu^1(\Omega)$, and that such minimizers are also action ground states was recently proved in \cite[Theorem 1.3]{DST}.
\end{remark}

\begin{remark} Combining our results with those of \cite{NTV1}, we obtain a perhaps unexpected consequence. When $\Omega$ is a ball, $p = 2+4/N$, and for $\mu \in [\mu_N, 2\mu_N)$ there exist least energy normalized nodal solutions with mass $\mu$, by Theorem  \ref{thm:LENS}. By \cite[Theorem 1.5]{NTV1}, there are no positive solutions of mass $\mu$. This means that {\em least energy solutions of mass $\mu$ are nodal}.

\end{remark}

Theorem \ref{thm:LENS} (and its counterpart for positive solutions) gives no insight in the $L^2$-supercritical regime. However, it makes perfect sense to wonder whether normalized solutions with minimal energy are action ground states also when $p>2+\frac4N$. At present, we are not able to answer this question for any general bounded and open set $\Omega$ in $\R^N$, but we can partially solve the problem at least for star-shaped domains of $\R^N$. The next theorem summarizes our results in this direction, that provide our main contribution with respect to question 1) above.
\begin{theorem}
	\label{thm:super}
	Let $\Omega$ be bounded, open, smooth and star-shaped, $p\in\left(
	2+\frac4N,2^*\right)$, and $\mu_p,\mu_p^{nod}$ be as in Theorem \ref{thm:masses}. Then there exists a least energy normalized solution for every $\mu\leq\mu_p$, and there exists a least energy normalized nodal solution for every $\mu\leq\mu_p^{nod}$. Moreover, there exist $\overline{\mu}_p\leq\mu_p$ and $C_p>0$ such that every least energy normalized solution $u$ with mass $\mu\leq\overline{\mu}_p$ is an action ground state in $\NN_\lambda(\Omega)$, where $\lambda$ is the number associated to $u$ in \eqref{nlse} and satisfies $\lambda\leq C_p$. Analogously, there exist $\overline{\mu}_p^{nod}\leq\mu_p^{nod}$ and $C_p^{nod}>0$ such that every least energy normalized nodal solution $v$ with mass $\mu\leq\overline{\mu}_p^{nod}$ is a nodal action ground state in $\NN_{\lambda'}^{nod}(\Omega)$, where $\lambda'$ is the number associated to $v$ in \eqref{nlse} and satisfies $\lambda'\leq C_p^{nod}$.
\end{theorem}
Note that Theorem \ref{thm:super} not only shows that, in certain regimes of masses, least energy solutions are again action ground states, but it also proves that the corresponding frequency $\lambda$ of such ground states is bounded from above uniformly in $\lambda$. In fact, in the $L^2$-supercritical regime it is not difficult to show that action ground states have small masses both when $\lambda$ is close to $-\lambda_1$  and when $\lambda$ is large. Theorem \ref{thm:super} suggests that, even though they have the same masses, small frequency ground states are energetically convenient. This kind of property had been observed before when $\Omega$ is a ball (see \cite[Theorem 1.7 and Remark 6.4]{NTV1}), and to some extent one can interpret Theorem \ref{thm:super} as a first step towards a proof of a result of this sort for general domains.

It is an open problem to understand whether the content of Theorem \ref{thm:super} remains true when $\Omega$ is not star-shaped. Note that, in the proof of Theorem \ref{thm:super} reported in Section \ref{sec:super} below, the star-shapedness assumption plays a role not only to show that action ground states are least energy normalized solutions, but also to prove that a least energy normalized solution actually exists.

To conclude, we wish to point out that the argument developed in the present paper is not limited to NLS equations \eqref{nlse} with a pure power nonlinearity and homogeneous Dirichlet conditions at the boundary. On the contrary, since it can be  generalized to other boundary conditions or nonlinearities, the paper actually provides a new approach to the study of normalized solutions of NLS equations, that one can try to exploit whenever a suitable Nehari manifold associated to the problem under exam is available. Moreover, this work can be seen as a further step in the investigation of the relation between the action and the energy approaches to the search of solutions of NLS equations, thus extending the first analyses in this direction recently started in \cite{DST,JL}. 

\smallskip
The remainder of the paper is organized as follows. Section \ref{sec:prel} recalls some known facts and proves preliminary existence results for nodal action ground states. Section \ref{sec:lev} provides a detailed analysis of the nodal action ground state level $\JJ_\Omega^{nod}$, whereas Section \ref{sec:masses} gives the proof of Theorem \ref{thm:masses} on the masses of action ground states. Finally, Section \ref{sec:super} completes the proof of the main results of the paper, namely Theorems \ref{thm:exnod}--\ref{thm:LENS}--\ref{thm:super}.

\smallskip
\noindent {\bf Notation.} Throughout, we will use shorter notations for norms as $\|u\|_q$, avoiding to write the domain of integration whenever it is clear by the context.

\section{Existence results for nodal action ground states}
\label{sec:prel}

This section discusses existence and non-existence of nodal action ground states on open bounded subsets $\Omega$ of $\R^N$. We recall that, if $\Omega$ is smooth, existence of such states has been proved in \cite[Theorem 1.1]{BW}. Hence, here we limit ourselves to prove some basic estimates that allow us to extend this already known result to general open and bounded sets (without regularity assumptions).

We start by recalling the picture for action ground states.
The following result concerning the action was proved in \cite[Theorem 1.5, Lemma 2.4 and Remark 2.5]{DST}. Here, $\lambda_1$ denotes the first eigenvalue of $-\Delta$ with homogeneous Dirichlet conditions at the boundary of $\Omega$.

\begin{proposition}
\label{prop_J}
For every $p\in (2,2^*)$, the following properties hold.
\begin{itemize}
\item[(i)] For every $\lambda \le -\lambda_1$, $\JJ_\Omega(\lambda) = 0$ and action ground states in $\NN_{\lambda}(\Omega)$ do not exist.
\item[(ii)] For every $\lambda > -\lambda_1$, $\JJ_\Omega(\lambda)>0$ and action ground states in $\NN_{\lambda}(\Omega)$  exist.
\item[(iii)] The function $\JJ_\Omega:\R \to \R$ is locally Lipschitz continuous and strictly increasing on $[-\lambda_1, +\infty)$.
\item[(iv)] Letting $Q_p(\lambda) = \left\{\|u\|_2^2 : u\in \NN_{\lambda}(\Omega) \text{ and } J_\lambda(u,\Omega)=\JJ_\Omega(\lambda) \right\}$, there results
\begin{equation*}
\lim_{\varepsilon\to0^+} \frac{\JJ_\Omega(\lambda+\varepsilon)-\JJ_\Omega(\lambda)}{\varepsilon} = \frac12 \inf Q_p(\lambda)
\le \frac12 \sup Q_p(\lambda) = \lim_{\varepsilon\to0^-} \frac{\JJ_\Omega(\lambda+\varepsilon)-\JJ_\Omega(\lambda)}{\varepsilon},
\end{equation*}
Moreover, for every $\lambda$ outside an at most countable set, all action ground states have the same mass (i.e., $Q_p(\lambda)$ is a singleton).
\end{itemize}
\end{proposition}

\begin{remark}
\label{no_pos_sol}
It is well known that the threshold $-\lambda_1$ appearing in the preceding result is also a threshold for the existence of constant sign solutions. Precisely, if $\lambda \le -\lambda_1$ then \eqref{nlseNOMASS} has no nonzero solutions $u$ with $u \ge 0$.
It is also well known that  if $\lambda > -\lambda_1$ and $u$ is a nonzero solution with $u \ge 0$, then $u > 0$ in $\Omega$.
\end{remark}

We now establish a similar picture for nodal action ground states. In this setting, a major role is played by the second eigenvalue $\lambda_2$ of $-\Delta$ with homogeneous Dirichlet conditions at the boundary. Requiring only $\lambda > -\lambda_2$ poses some problems that one does not encounter in the study of signed ground states. For instance, the inequality $\|\nabla u\|_2^2+\lambda\|u\|_2^2> 0$ does not hold for every $u$ and checking it  for a given $u$ requires some care. Also, we will have to estimate the norms of the positive and negative parts of functions separately, something that is not directly readable from the functional $J_\lambda$. For these reasons we proceed with single statements instead of collecting them all together as in Proposition \ref{prop_J}.

We recall that, given $\lambda\in \mathbb R$ and $u\in H_0^1(\Omega)$ such that $\|\nabla u\|_2^2+\lambda\|u\|_2^2> 0$, defining
\begin{equation*}
    n_\lambda(u):=\left(\f{\|\nabla u\|_2^2+\lambda\|u\|_2^2}{\|u\|_p^p}\right)^{\f 1{p-2}},
\end{equation*}
we have $n_\lambda(u)u\in\NL(\Omega)$. We also recall that if $u \in \NN_{\lambda}(\Omega)$, then
\begin{equation}
    \label{form_J}
    J_{\lambda}(u,\Omega)
    = \kappa \| u \|_p^p
    = \kappa \Bigl( \| \nabla u \|_2^2 + \lambda \| u \|_2^2 \Bigr), \qquad \kappa = \frac12 - \frac1p,
\end{equation}
a fact we will tacitly use in the proofs.
\begin{remark}
	\label{rem:NAGS_sol}
	The fact that nodal action ground states in $\NN_\lambda^{nod}(\Omega)$, when they exist, are solutions of problem \eqref{nlseNOMASS} is well known (see e.g. \cite[Proposition 3.1]{BWW}).
\end{remark}

The next proposition is the nodal analogue of Proposition \ref{prop_J}$(i)$.

\begin{proposition}
\label{properties_NGS}
For every $p\in (2,2^*)$ and every $\lambda \le -\lambda_2$, $\JJ_\Omega^{nod}(\lambda) = 0$ and nodal action ground states in $\NN_\lambda^{nod}(\Omega)$ do not exist.
\end{proposition}

\begin{proof} Fix any $\lambda \le -\lambda_2$ and let $\varphi_2 \in H^1_0(\Omega)$
be an eigenfunction corresponding to $\lambda_2=\lambda_2(\Omega)$.
Denoting by $\Omega^+ := \hbox{supp} (\varphi_2^+)$ and $\Omega^- :=\hbox{supp}  (\varphi_2^-)$,
there results, as is well known,  $\lambda_2(\Omega) = \lambda_1(\Omega^+) = \lambda_1(\Omega^-)$.
Then, by assumption, $\lambda \le -\lambda_1(\Omega^+) = -\lambda_1(\Omega^-)$.
By Proposition \ref{prop_J}$(i)$ we deduce that
\[
\JJ_\Omega^{nod}(\lambda) \le \inf_{v \in \NN_{\lambda}(\Omega^+)} J_{\lambda}(v,\Omega^+) + \inf_{v \in \NN_{\lambda}(\Omega^-)} J_{\lambda}(v,\Omega^-) = 0.
\]
Since $\JJn(\lambda)$ is always nonnegative by \eqref{form_J},  we see that $\JJ^{nod}_\Omega(\lambda)=0$. Furthermore, again by \eqref{form_J}, if $u$ were a nodal ground state in $\NN_\lambda^{nod}(\Omega)$, we would have $\| u \|_p=0$, which is impossible as $0\not\in\NN_{\lambda}^{nod}(\Omega)$.
\end{proof}

\bigskip
In the following, let
\begin{equation*}
 C(p) := \inf_{u \in H^1_0(\Omega)\setminus\{0\}} \frac{\| \nabla u \|_2}{\| u \|_p}.
\end{equation*}
\begin{proposition}
\label{prelpropJ} \mbox{}
For every $p\in (2,2^*)$, there exist positive constants $C_1,C_2$ such that for every $\lambda \ge  -\lambda_2$,
\begin{align}
 \label{up}
&\JJn(\lambda) \le C_1  (\lambda + \lambda_2)^{\frac{p}{p-2}}\\
\label{down}
 &\JJn(\lambda) \ge C_2 \min\left(1,  \frac{\lambda + \lambda_2}{\lambda_2}\right)^{\frac{p}{p-2}} 
\end{align}
\end{proposition}

\begin{proof} 

To prove \eqref{up}, notice that when $\lambda > -\lambda_2$, if $\varphi_2\in H_0^1(\Omega)$ is an
 eigenfunction  associated to $\lambda_2$, 
then $\|\nabla \varphi_2^\pm\|_2^2 +\lambda \|\varphi_2^\pm\|_2^2 = (\lambda +\lambda_2)\| \varphi_2^\pm\|_2^2 > 0$, so that $n_\lambda(\varphi_2^\pm)$ is well defined and $n_{\lambda}(\varphi_2^+)\varphi_2^+ + n_{\lambda}(\varphi_2^-)\varphi_2^-  \in \NN_\lambda^{nod}(\Omega)$. Hence,
 \begin{align*}
\JJn(\lambda) &\le  J_{\lambda}\bigl(n_{\lambda}(\varphi_2^+)\varphi_2^+ + n_{\lambda}(\varphi_2^-)\varphi_2^-,\Omega \bigr) =\kappa n_\lambda(\varphi_2^+)^p \|\varphi_2^+\|_p^p +\kappa n_\lambda(\varphi_2^-)^p \|\varphi_2^-\|_p^p \\
 &= \kappa ( \lambda+ \lambda_2)^\frac{p}{p-2} \left( \left( \frac{\| \varphi_2^+ \|_2}{\| \varphi_2^+ \|_p} \right)^\frac{2p}{p-2} + 
 \left( \frac{\| \varphi_2^- \|_2}{\| \varphi_2^- \|_p} \right)^\frac{2p}{p-2} \right) =: C_1  ( \lambda+ \lambda_2 )^\frac{p}{p-2},
 \end{align*}
which is \eqref{up}. To prove \eqref{down} we use an argument taken from \cite{BBGVS}. Given $u \in \NN_{\lambda}^{nod}(\Omega)$, we notice that, plainly, there exists $s \in (0, 1)$ such that $u_s := s u^+ + (1-s)u^-$ is $L^2$--orthogonal to the eigenspace $E_1$ associated with $\lambda_1$.
Since $n_{\lambda}(u_s)$ is well defined, we see that the function 
$v := n_{\lambda}(u_s)u_s$ belongs to $\NN_{\lambda}(\Omega) \cap E_1^\perp$, and we write it as $v = \alpha u^+ + \beta u^-$ for some $\alpha, \beta > 0$.
Then, as $u^+$ and $u^-$ belong to $\NN_{\lambda}(\Omega)$, 
\begin{equation}
\label{level_u_v}
J_{\lambda}(v,\Omega) = J_{\lambda}(\alpha u^+,\Omega) + J_{\lambda}(\beta u^-,\Omega) \le J_{\lambda}(u^+,\Omega) + J_{\lambda}(u^-,\Omega) = J_{\lambda}(u,\Omega),
\end{equation}
since, by definition of Nehari manifold, $J_\lambda(tu^\pm,\Omega) \le J_\lambda(u^\pm,\Omega)$ for every $t >0$. Now, if $\lambda \ge 0$, obviously
\[
\| \nabla v \|_2^2 \le \| \nabla v \|_2^2 +\lambda \| v \|_2^2,
\]
while, if $\lambda \in (-\lambda_2,0)$,
\[
\frac{\lambda + \lambda_2}{\lambda_2}  \| \nabla v \|_2^2 =  \| \nabla v \|_2^2 + \frac{\lambda}{\lambda_2} \| \nabla v \|_2^2 \le  
\| \nabla v \|_2^2 + \lambda \|  v \|_2^2
\]
because $v \in E_1^\perp$. From the two preceding inequalities we obtain
\[
\min\left(1, \frac{\lambda + \lambda_2}{\lambda_2} \right)\| \nabla v \|_2^2 \le  \| \nabla v \|_2^2 + \lambda \| v \|_2^2 
= \| v \|_p^p \le C(p)^{-p} \| \nabla v \|_2^p,
\]
so that
\[
\| \nabla v \|_2 \ge C(p)^\frac{p}{p-2} \min\Bigl(1, \frac{\lambda + \lambda_2}{\lambda_2} \Bigr)^\frac{1}{p-2}.
\]
Thus
\[
J_{\lambda}(v,\Omega) = \kappa \Bigl( \| \nabla v \|_2^2 + \lambda \| v \|_2^2 \Bigr)
\ge \kappa \min\Bigl(1, \frac{\lambda + \lambda_2}{\lambda_2} \Bigr)\| \nabla v \|_2^2
\ge \kappa C(p)^\frac{2p}{p-2} \min\Bigl(1, \frac{\lambda + \lambda_2}{\lambda_2} \Bigr)^\frac{p}{p-2},
\]
which proves \eqref{down} (with $C_2 =  \kappa C(p)^\frac{2p}{p-2}$) using \eqref{level_u_v} and taking the infimum over $u \in \NN_\lambda^{nod}(\Omega)$.

\end{proof}

\begin{remark}
For future reference, we notice that an estimate similar to \eqref{down} holds for $\JJ_\Omega$: there exists $C>0$ such that for every $\lambda \ge -\lambda_1$,
\begin{equation*}
\JJ_\Omega(\lambda) \ge C \min\left(1,  \frac{\lambda + \lambda_1}{\lambda_1}\right)^{\frac{p}{p-2}}.
\end{equation*}
To check this, it is enough to notice that, for every $u \in \NN_\lambda(\Omega)$ with $\lambda \in (-\lambda_1, 0)$, 
\[
\frac{\lambda + \lambda_1}{\lambda_1}  \| \nabla u \|_2^2 =  \| \nabla u \|_2^2 + \frac{\lambda}{\lambda_1} \| \nabla u \|_2^2 \le  
\| \nabla u \|_2^2 + \lambda \|  u \|_2^2
\]
and proceed exactly as in the proof of Proposition \ref{prelpropJ}.
\end{remark}

The following ``a priori'' type result will be used in the proof of existence of nodal action ground states on arbitrary domains and will also provide a useful tool for the next sections.

\begin{proposition}
\label{convergence_NGS}
Let $(\Omega_n)_{n \ge 1}$ be a sequence of connected open sets such that
$\Omega_n \subseteq \Omega_{n+1}$ for all $n$
and $\Omega = \bigcup_{n \ge 1} \Omega_n$.
Let $(\alpha_n)_{n \ge 1} \subseteq (-\lambda_2, +\infty)$
be a sequence converging to $\lambda \in \intervaloo{-\lambda_2, +\infty}$, where $\lambda_2=\lambda_2(\Omega)$ is the second Dirichlet eigenvalue of $-\Delta$ on $\Omega$.
    
Suppose that, for every $n$, $u_n \in H_0^1(\Omega)$ 
is a nodal action ground state in $\NN_{\alpha_n}^{nod}(\Omega_n)$, extended by $0$ on $\Omega \setminus \Omega_n$.
Then, up to subsequences, $(u_n)_n$ converges in $H^1_0(\Omega)$ to a function $u$, which is
a nodal action ground state in $\NN_{\lambda}^{nod}(\Omega)$ and
\begin{equation}
\label{limitJn}
 \lim_{n \rightarrow \infty} \JJ_{\Omega_n}^{nod}(\alpha_n) = \JJn(\lambda) = J_\lambda(u,\Omega).
\end{equation}
\end{proposition}

\begin{proof}
Let $\varphi_2 \in H^1_0(\Omega_1)$ be an eigenfunction corresponding to $\lambda_2(\Omega_1) \ge \lambda_2(\Omega)$,
 extended by $0$ to $\Omega$. Since for every $n\ge 1$
\[
\|\nabla \varphi_2^\pm\|_2^2 + \alpha_n\| \varphi_2^\pm\|_2^2 = (\lambda_2(\Omega_1)  + \alpha_n)\| \varphi_2^\pm\|_2^2 > 
(\lambda_2(\Omega_1)  -\lambda_2(\Omega))\| \varphi_2^\pm\|_2^2 \ge 0,
\]
we see that the numbers $n_{\alpha_n}(\varphi_2^\pm)$ are well defined. Therefore
\begin{align*}
J_{\alpha_n}(u_n,\Omega) &= \inf_{v \in \NN_{\alpha_n}^{nod}(\Omega_n)} J_{\alpha_n}(v,\Omega_n)
\le J_{\alpha_n}\bigl(n_{\alpha_n}(\varphi_2^+)\varphi_2^+ + n_{\alpha_n}(\varphi_2^-)\varphi_2^- ,\Omega_n\bigr)\\
&=\kappa\bigl( \lambda_2(\Omega_1) + \alpha_n \bigr)^\frac{p}{p-2}
\Biggl( \biggl( \frac{\| \varphi_2^+ \|_2}{\| \varphi_2^+ \|_p} \biggr)^\frac{2p}{p-2} +
\biggl( \frac{\| \varphi_2^- \|_2}{\| \varphi_2^- \|_p} \biggr)^\frac{2p}{p-2} \Biggr),
\end{align*}
from which we deduce (the sequence $\alpha_n$ being convergent) that $J_{\alpha_n}(u_n,\Omega)$ is bounded. By \eqref{form_J}, this implies that
 $u_n$ is bounded in $L^p(\Omega)$ and hence in $L^2(\Omega)$, because $\Omega$ is bounded.
Again by \eqref{form_J}, we obtain that $u_n$ is bounded in $H^1_0(\Omega)$.
Therefore, up to  subsequences, $(u_n)_n$ converges weakly in $H^1_0(\Omega)$
and strongly in $L^p(\Omega)$ to $u \in H^1_0(\Omega)$. Noticing that, if $u_n \in \NN_{\alpha_n}^{nod}(\Omega_n)$, then we also have $u_n \in \NN_{\alpha_n}^{nod}(\Omega)$, \eqref{down} shows that
\begin{align*}
\kappa\|u\|_p^p &= \kappa \lim_n \|u_n \|_p^p = \lim_n J_{\alpha_n}(u_n,\Omega_n) \ge \liminf_n \JJn(\alpha_n) \\
&\ge C_2\liminf_n  \min\left(1,  \frac{\alpha_n + \lambda_2}{\lambda_2}\right)^{\frac{p}{p-2}} = C_2\min\left(1,  \frac{\lambda+ \lambda_2}{\lambda_2}\right)^{\frac{p}{p-2}} >0
\end{align*}
because $\alpha_n \to \lambda >-\lambda_2$. Therefore $u \not\equiv 0$.

Given any $\varphi \in \C_c^\infty(\Omega)$,
we have  $\supp(\varphi) \subseteq \Omega_n$ for all $n$ large enough, and since
 by Remark \ref{rem:NAGS_sol} $u_n$ is a solution of \eqref{nlseNOMASS} in $\Omega_n$ with multiplier $\alpha_n$,
\begin{equation*}
\int_{\Omega} \nabla u_n \cdot \nabla \varphi + \alpha_n \int_{\Omega} u_n \varphi = \int_{\Omega} |u_n|^{p-2} u_n \varphi
\end{equation*}
for all $n$ large enough. Letting $n \to \infty$, we see that $u$ solves problem \eqref{nlseNOMASS} with $\lambda = \lim_n \alpha_n$.
    
Next we show that $u$ is nodal. Assume by contradiction that, say,  $u \ge 0$. 
Then, by Remark~\ref{no_pos_sol},  $\lambda > -\lambda_1$.
Since, as above, $u_n^\pm \in \NN_{\alpha_n}(\Omega)$, by Proposition~\ref{prop_J} we see that
\[
\kappa\|u^-\|_p^p = \kappa \lim_n \|u_n^- \|_p^p = \lim_n J_{\alpha_n}(u_n^-,\Omega_n) \ge \liminf_n \JJ_\Omega(\alpha_n) = \JJ_\Omega(\lambda) > 0,
\]
i.e. a contradiction. Hence, $u\in\NN_\lambda^{nod}(\Omega)$.
    
It remains to prove that $u$ is a nodal action ground state and that \eqref{limitJn} holds. First notice that, by the already proved convergences of $u_n$ and $\alpha_n \to \lambda$,
\begin{equation}
\label{lim1}
\lim_n \JJ_{\Omega_n}^{nod}(\alpha_n) = \lim_n J_{\alpha_n}(u_n,\Omega_n) = \lim_n \kappa\|u_n\|_p^p = \kappa\|u\|_p^p = J_\lambda(u,\Omega) \ge \JJn(\lambda),
\end{equation}
because $u_n \in \NN_{\alpha_n}^{nod}(\Omega_n)$ and $u \in  \NN_\lambda^{nod}(\Omega)$.
   
We now establish the reversed inequality. To this aim, given any $\varepsilon >0$, it is easily seen, by density, that there exists a function $v \in \NN_\lambda^{nod}(\Omega) \cap C_c^\infty(\Omega)$ such that $J_\lambda(v,\Omega) \le \JJn(\lambda) + \eps$. Since $v$ has compact support in $\Omega$, actually $v \in  \NN_\lambda^{nod}(\Omega_n)$ for every $n$ large enough. Now, as $n\to \infty$,
\[
n_{\alpha_n} (v^+)^{p-2} = \frac{\|\nabla v^+\|_2^2 +\alpha_n\|v^+\|_2^2}{\|v^+\|_p^p} = 1 + (\alpha_n -\lambda)\frac{\|v^+\|_2^2}{\|v^+\|_p^p} = 1 +o(1),
\]
and the same for $n_{\alpha_n} (v^-)$. In particular, $n_{\alpha_n} (v^+)v^+ +n_{\alpha_n} (v^-)v^- \in \NN_{\alpha_n}(\Omega_n)$ and 
\begin{align*}
\JJ_{\Omega_n}^{nod}(\alpha_n) &\le J_{\alpha_n}\left(n_{\alpha_n} (v^+)v^+ +n_{\alpha_n} (v^-)v^-, \Omega_n \right) = \kappa n_{\alpha_n} (v^+)^p \|v^+\|_p^p +
\kappa n_{\alpha_n} (v^-)^p \|v^-\|_p^p \\
&\le (1+o(1))\kappa \|v\|_p^p = (1+o(1)) J_\lambda(v,\Omega) \le (1+o(1))\left(\JJ_\Omega^{nod}(\lambda) + \eps\right),
\end{align*}
from which we obtain
\[
\lim_n \JJ_{\Omega_n}^{nod}(\alpha_n) \le \JJ_\Omega^{nod}(\lambda) + \eps.
\]
Recalling that $\eps$ is arbitrary and coupling with \eqref{lim1}, we see that $u$ is a nodal action ground state and \eqref{limitJn} holds. This also shows that the convergence of $u_n$ to $u$ is strong in $H_0^1(\Omega)$ and completes the proof.
\end{proof}

We can now prove that nodal action ground states exist on arbitrary bounded open sets $\Omega$
when $\lambda > -\lambda_2(\Omega)$.

\begin{proposition}
\label{Prop Nodal}
For every $p \in (2,2^*)$ and every $\lambda>-\lambda_2$,  nodal action ground states in $\NN_\lambda^{nod}(\Omega)$ exist.
\end{proposition}

\begin{proof}
A proof that  nodal action ground states exist when $\lambda > -\lambda_1$ can be found in \cite{CCN, SW}.
This result was extended in \cite[Theorem 1.1]{BW} to cover all $\lambda > -\lambda_2$, assuming $\Omega$ is smooth
(regularity is used to turn $\NN_{\lambda}^{nod}(\Omega)$ into a manifold,
see \cite[Lemma 3.2]{BW}). Here we slightly extend \cite[Theorem 1.1]{BW} to arbitrary domains.

Using \cite[Proposition 8.2.1]{Da}, there exists a sequence of connected and bounded
open sets $\Omega_n$ with smooth boundary such that $\Omega_n \subseteq \Omega_{n+1}$
for all $n$ and $\Omega = \bigcup_{n \ge 1} \Omega_n$.  Let $\lambda > -\lambda_2(\Omega)$.
By inclusion, $\lambda_2(\Omega_n) \ge \lambda_2(\Omega)$ for all $n \ge 1$,
so that $\lambda > -\lambda_2(\Omega_n)$. By \cite[Theorem 1.1]{BW},
there exists a nodal action ground state $u_n\in \NN_\lambda^{nod}(\Omega_n)$.
Applying Proposition~\ref{convergence_NGS}, the sequence $(u_n)_n$ converges, up to subsequences,
to a nodal action ground state in $\NN_\lambda^{nod}(\Omega)$.
\end{proof}

\section{The level of nodal action ground states}
\label{sec:lev}

In this section we collect some properties of the nodal action ground state level that will be used later on. 
We begin by a lower bound on the $L^p$ norm of positive and negative parts of nodal action ground states.

\begin{lemma}
    \label{lower_bound_parts}
For every  $\alpha > -\lambda_2$ there exists $c_\alpha > 0$ such that, for all $\lambda \ge \alpha$
and all $u \in \NN_{\lambda}^{nod}(\Omega)$ such that $J_{\lambda}(u,\Omega) = \JJ_\Omega^{nod}(\lambda)$, there results
\begin{equation}
\label{calpha}
\| u^\pm \|_p \ge c_\alpha.
\end{equation}
\end{lemma}

\begin{proof} We argue by contradiction. Assume then that there exist sequences $(\lambda_n)_n$ and $(u_n)_n$ such that, for every $n$, $\lambda_n\geq\alpha$ and  $u_n \in \NN_{\lambda_n}^{nod}(\Omega)$ is a nodal action ground state such that, for instance,
    \begin{equation}
        \label{lim_u_n_plus_contrad}
        \| u_n^+ \|_p \to 0
\end{equation}
as $n \to \infty$. Observe that this implies that $\JJ_\Omega(\lambda_n)\to0$ as $n\to\infty$, so that $\limsup_n \lambda_n\leq-\lambda_1$ by Proposition \ref{prop_J}. Hence, $(\lambda_n)_n$ is bounded and, up to subsequences, we may assume that $\lambda_n \to a$ for some $a \ge \alpha$. In particular, $a > -\lambda_2$. Then, Proposition~\ref{convergence_NGS} implies that, up to subsequences again, $(u_n)_n$ converges in $H^1(\Omega)$ to some nodal ground state $u$ in $\NN_a^{nod}(\Omega)$. In particular, $u^+ \ne 0$. Since $(u_n^+)_n$ converges strongly in $L^p(\Omega)$ to $u^+$,
    this contradicts \eqref{lim_u_n_plus_contrad}.
\end{proof}

We can now provide an analogue of Proposition \ref{prop_J} in the nodal setting.

\begin{proposition}
\label{deraction} For every $p \in (2,2^*)$, 
\begin{itemize}
\item[i)] the function $\JJn:\R \to \R$ is locally Lipschitz continuous and strictly increasing on $[-\lambda_2, +\infty)$;
\item[ii)] for every $\lambda>-\lambda_2$, let
\[
Q_p^{nod}(\lambda) := \Bigl\{\|u\|_2^2 \mid u\in \NN_{\lambda}^{nod}(\Omega)
\text{ and } J_\lambda(u,\Omega)=\JJ_\Omega^{nod}(\lambda) \Bigr\}.
\]
Then
\end{itemize}
\begin{equation}
\label{der}
\limsup_{\varepsilon\to0^+} \frac{\JJ_\Omega^{nod}(\lambda+\varepsilon)-\JJ_\Omega^{nod}(\lambda)}{\varepsilon} \le
\frac12 \inf Q_p^{nod}(\lambda)  \le \frac12 \sup Q_p^{nod}(\lambda) \le
 \liminf_{\varepsilon\to0^-} \frac{\JJ_\Omega^{nod}(\lambda+\varepsilon)-\JJ_\Omega^{nod}(\lambda)}{\varepsilon}.
\end{equation}
Furthermore, for almost every $\lambda$, all nodal action ground states have the same mass (i.e. $Q_p^{nod}(\lambda)$ is a singleton).
\end{proposition}

\begin{proof}
To establish $i)$, note first that the continuity of $\JJn$ follows directly by Proposition \ref{convergence_NGS} taking $\Omega_n=\Omega$ for every $n$. 

Since $\JJn \equiv 0$ on $(-\infty, -\lambda_2]$ by Proposition \ref{properties_NGS}, to obtain the desired monotonicity of $\JJn$ we only have to prove that $\JJn$ is strictly increasing for $\lambda \ge -\lambda_2$.
To check this, it is enough to show that every $\lambda >-\lambda_2$ has a neighborhood where $\JJn$ is strictly increasing. To this aim, let $\lambda >-\lambda_2$ be arbitrarily fixed. We claim that there exists $\delta >0$ such that, for every $\alpha, \beta \in (\lambda-\delta, \lambda+\delta)$ with $\alpha < \beta$, if $u_\beta \in \NN_\beta^{nod}(\Omega)$ satisfies $J_\beta (u_\beta,\Omega) = \JJ_\Omega^{nod}(\beta)$, then $\|\nabla u_\beta^\pm \|_2^2 + \alpha \|u_\beta^\pm\|_2^2 >0$. Indeed,
since $\Omega$ is bounded and $p>2$, $\|u\|_2 \le C\|u\|_p$ for every $u \in H^1_0(\Omega)$, with $C = |\Omega|^{\frac{p-2}{2p}}$. Hence,
\begin{align*}
\|\nabla u_\beta^\pm \|_2^2 + \alpha \|u_\beta^\pm\|_2^2 &= \|\nabla u_\beta^\pm \|_2^2 + \beta \|u_\beta^\pm\|_p^p +  (\alpha -\beta) \|u_\beta^\pm\|_2^2 =
\|u_\beta^\pm\|_p^p + (\alpha -\beta) \|u_\beta^\pm\|_2^2 \\
& \ge \|u_\beta^\pm\|_p^p + C^2(\alpha -\beta) \|u_\beta^\pm\|_p^2
>0 
\end{align*}
if $\beta -\alpha$ (namely $\delta$) is small enough (note that, by Lemma \ref{lower_bound_parts},  $\|u_\beta^\pm\|_p$ are uniformly bounded away from zero if $\beta$ is bounded away from $-\lambda_2$). This shows that $n_\alpha(u_\beta^\pm)$ is well defined. So, letting $\alpha, \beta$ and $u_\beta$ be as above and noticing that $n_\alpha (u_\beta^\pm) < 1$, we have
\begin{align}
\label{interm}
\JJ^{nod}_\Omega(\alpha) 
&\le
J_{\alpha}\left(n_{\alpha}(u_\beta^+)u_\beta^+ + n_{\alpha}(u_\beta^-)u_\beta^-,\Omega\right)
= \kappa \Big(n_{\alpha}(u_\beta^+)^p\|u_\beta^+\|_p^p + n_{\alpha}(u_\beta^-)^p\|u_\beta^-\|_p^p \Big)
\nonumber \\
&<  \kappa \|u_\beta\|_p^p  = J_{\beta}(u_\beta,\Omega)= \JJ_\Omega^{nod}(\beta),
\end{align}
which shows that $\JJ_\Omega^{nod}$ is strictly increasing around every $\lambda > -\lambda_2$.

To complete the proof of $i)$, it remains to show that $\JJn$ is locally Lipschitz continuous. To this end, we first prove \eqref{der}.
For the first inequality, let $\lambda > -\lambda_2$ and let $u \in \NN_\lambda^{nod}(\Omega)$ be any function such that $J_\lambda(u,\Omega) = \JJn(\lambda)$ (at least one such nodal action ground state exists by Proposition \ref{Prop Nodal}). For every $\eps >0$, we have, as in \eqref{interm},
\begin{align*}
\JJ_\Omega^{nod}(\lambda+\varepsilon)-\JJ_\Omega^{nod}(\lambda)
&\le
J_{\lambda+\varepsilon}(n_{\lambda+\varepsilon}(u^+)u^+ + n_{\lambda+\varepsilon}(u^-)u^-,\Omega)-J_{\lambda}(u,\Omega)
\\
&= \kappa\Big[\Big(n_{\lambda+\varepsilon}(u^+)^p-1\Big)\|u^+\|_p^p + \Big(n_{\lambda+\varepsilon}(u^-)^p-1\Big)\|u^-\|_p^p \Big].
\end{align*}
Since, as $\eps \to 0^+$,
\[
n_{\lambda+\varepsilon}(u^\pm)^p = 1 + \frac{\eps p}{p-2} \frac{\|u^\pm\|_2^2}{\|u^\pm\|_p^p} + o(\eps),
\]
we have
\begin{align*}
\JJ_\Omega^{nod}(\lambda+\varepsilon)-\JJ_\Omega^{nod}(\lambda) &\le \kappa \left( \frac{\eps p}{p-2}\frac{\|u^+\|_2^2}{\|u^+\|_p^p} + o(\eps)\right)\|u^+\|_p^p + \kappa \left( \frac{\eps p}{p-2}\frac{\|u^-\|_2^2}{\|u^-\|_p^p} + o(\eps)\right)\|u^-\|_p^p \\
&= \frac\eps2 \|u\|_2^2 +o(\eps).
\end{align*}
Therefore, for all $u\in \NN^{nod}_\lambda(\Omega)$,
\[
\limsup_{\varepsilon\to0^+}
\frac{\JJ_\Omega^{nod}(\lambda+\varepsilon)-\JJ_\Omega^{nod}(\lambda)}{\varepsilon}
\le \frac12 \|u\|_2^2,
\]
proving the first part of \eqref{der}.

The argument for the last inequality in \eqref{der} is the same, after noticing that $n_{\lambda+\eps}(u) >0$ for every negative $\eps$ small enough.

We can now prove that $\JJn$ is locally Lipschitz continuous on $[-\lambda_2,+\infty)$ (the claim is trivial on $(-\infty,-\lambda_2]$). Note first that, by the first inequality in \eqref{der}, the fact that $\|u\|_2^p\leq C \|u\|_p^p=2pC\JJn(\lambda)/(p-2)$ if $u$ is a nodal action ground state in $\NN_\lambda^{nod}(\Omega)$ and the continuity of $\JJn$, for every compact interval $K\subset[-\lambda_2,+\infty)$ there exists a constant $L>0$ such that
\begin{equation}
\label{limsup}
\sup_{\lambda\in K}\limsup_{\varepsilon\to0^+}\frac{\JJn(\lambda+\varepsilon)-\JJn(\lambda)}{\varepsilon}\leq\frac L2.
\end{equation}
Then, given $\alpha\in K$, consider the function $f(\beta):=\JJn(\beta)-\JJn(\alpha)-L(\beta-\alpha)$, defined for every $\beta \in K$ such that $\beta\geq\alpha$. Assume by contradiction that there exists $\beta$ such that $f(\beta)>0$. Since, by definition of $f$ and $L$, $f(s)$ is strictly smaller than $f(\alpha)=0$ for every $s$ in a right neighborhood of $\alpha$, then $f$ would have a minimum point $\bar\beta$ in the interior of $K$. But this is impossible, since at $\bar\beta$ it should be $\displaystyle \limsup_{\varepsilon\to0^+}\frac{\JJn(\bar\beta+\varepsilon)-\JJn(\bar\beta)}\varepsilon\geq L$, contradicting \eqref{limsup}. Hence, $f$ is non-positive on $K$, that is $\JJn$ is $L$-Lipschitz continuous on $K$. This completes the proof of $i)$.

Finally, since $\JJ_\Omega^{nod}$ is locally Lipschitz continuous,
it is in particular differentiable for almost every $\lambda > -\lambda_2$. If $\JJ_\Omega^{nod}$ is differentiable at some $\lambda$, all inequalities in \eqref{der} are equalities,
and $Q_p^{nod}(\lambda)$ is a singleton, concluding the proof of $ii)$.
\end{proof}

We now turn to the asymptotic behavior of $\JJn(\lambda)$.

\begin{proposition}
\label{lem:J}
Let $\mu_N$ be the number defined in \eqref{eq:muN}. Then
\begin{equation}
\label{J/lambda}
\lim_{\lambda\to+\infty}\frac{\JJ_\Omega^{nod}(\lambda)}{\lambda}=
\begin{cases}
+\infty & \text { if }p\in\left(2,2+\frac4N\right),
\\
\mu_N & \text{ if }p=2+\frac4N,
\\
0 & \text { if }p\in\left(2+\frac4N,2^*\right).
\end{cases}
\end{equation}
\end{proposition}

\begin{proof}
It is just a slight variant of \cite[Lemma 2.4]{DST}, where the analogous properties are proved for $\JJ_\Omega$. We will therefore be rather sketchy here.  For every $\mu>0$, set
\[
\EE_\Omega(\mu):=\inf_{u\in\Hmu(\Omega)}E(u,\Omega).
\]
In \cite[Proposition 2.3]{DST}, it has been shown that for every $p\in (2,2^*)$, every $\lambda \in \R$ and every $\mu>0$,
\[
\JJ_\Omega(\lambda)\ge \ee_\Omega(2\mu)+\lambda\mu
\]
(be careful that in \cite{DST} the manifold $\MM_\mu(\Omega)$ is defined as the set of functions in $H_0^1(\Omega)$ whose $L^2$ norm is equal to $2\mu$, whereas here functions in $\MM_\mu(\Omega)$ have $L^2$ norm equal to $\mu$).
If $u$ is any element in $\NN_{\lambda}^{nod}(\Omega)$,
then $u^\pm \in \NN_{\lambda}(\Omega)$ and, by the previous inequality,
\[
J_\lambda(u,\Omega) = J_\lambda(u^+,\Omega)+J_\lambda(u^-,\Omega) \ge 2\JJ_\Omega(\lambda) \ge 2(\ee_\Omega(2\mu)+\lambda\mu).
\]
Taking the infimum over $u \in \NN_{\lambda}^{nod}(\Omega)$ we conclude that for every $p\in (2,2^*)$, every $\lambda \in \R$ and every $\mu>0$,
\begin{equation}
\label{myth}
\JJn(\lambda)\ge 2\left(\ee_\Omega(2\mu)+\lambda\mu\right).
\end{equation}
We now distinguish cases according to the value of $p$.
\medskip

\noindent{\em Case 1: $p\in (2,2+4/N)$.} It is well known that in this case $\ee_\Omega(2\mu)$ is finite for every $\mu>0$. Therefore, by \eqref{myth}, 
\[
\liminf_{\lambda\to +\infty} \frac{\JJ_\Omega^{nod}(\lambda)}{\lambda} \ge \liminf_{\lambda\to +\infty} \frac{2(\ee_\Omega(2\mu) +\lambda\mu)}{\lambda} = 2\mu.
\]
Since $\mu$ is arbitrary, the conclusion follows.
\medskip
	
\noindent{\em Case 2: $p=2+4/N$.} By \cite[Lemma 2.1]{DST}, $\ee_\Omega(\mu_N) = 0$, so that again by \eqref{myth}, for every $\lambda > 0$, we have
\[
\JJ_\Omega^{nod}(\lambda)\ge 2\left(\ee(\mu_N)+\lambda\frac{\mu_N}2\right)=\lambda\mu_N,
\]
yielding $\JJ_\Omega^{nod}(\lambda)/\lambda\ge \mu_N$. Conversely, in \cite[Lemma 2.4]{DST} it is shown that, for sufficiently large $\lambda$, there exists a function $v_\lambda \in \NN_\lambda(\Omega)$ compactly supported in a small ball $B$ contained in $\Omega$ and such that $\|v_\lambda\|_2^2=\mu_N$ and $E(v_\lambda,B) / \lambda=o(1)$ as $\lambda \to +\infty$. Let $w_\lambda$ be a translation of $v_\lambda$  supported in another ball contained in $\Omega$ but disjoint from $B$. Clearly, $v_\lambda - w_\lambda \in \NN_\lambda^{nod}(\Omega)$,
$\|v_\lambda - w_\lambda\|_2^2=2\mu_N$ and $E(v_\lambda - w_\lambda,\Omega) / \lambda=o(1)$ as $\lambda \to +\infty$. Hence,
\[
\limsup_{\lambda\to+\infty}\frac{\JJ_\Omega^{nod}(\lambda)}\lambda\leq \limsup_{\lambda\to+\infty}\frac{J_\lambda(v_\lambda -w_\lambda,\Omega)}\lambda
=\lim_{\lambda\to+\infty}\frac{E(v_\lambda -w_\lambda,\Omega) + \lambda\mu_N}\lambda = \mu_N,
\]
and Case 2 is proved.
\medskip

\noindent{\em Case 3:  $p\in\left(2+4/N,2^*\right)$.} Again as in \cite[Lemma 2.4]{DST}, for large $\lambda$ we can construct a function $v_\lambda \in \NN_\lambda(\Omega)$ supported in a small ball $B \subset \Omega$ and such that $J(v_\lambda,\Omega) /\lambda \to 0 $ as $\lambda \to +\infty$ (thanks to $p > 2+4/N$). It suffices now to define $w_\lambda$ as in Case 2 to see that $v_\lambda-w_\lambda \in \NN_\lambda^{nod}(\Omega)$ and
\[
\limsup_{\lambda \to +\infty} \frac{\JJ_\Omega^{nod}(\lambda)}{\lambda} \le \lim_{\lambda \to +\infty} \frac{J_\lambda(v_\lambda -w_\lambda,\Omega)}{\lambda} =
2\lim_{\lambda \to +\infty} \frac{J_\lambda(v_\lambda,\Omega)}\lambda = 0.
\]
Since (for $\lambda >0$), $\JJ_\Omega^{nod}(\lambda)/\lambda \ge 0$ by \eqref{down}, the proof is complete.
\end{proof} 

\section{Proof of Theorem \ref{thm:masses}}
\label{sec:masses}

The purpose of this section is to prove the characterization of the masses of all action ground states and nodal action ground states as stated in Theorem \ref{thm:masses}. 

To this end, recall the definition of the sets $M_p(\Omega), M_p^{nod}(\Omega)$ given in \eqref{M}, and note that
\[
M_p(\Omega)=\bigcup_{\lambda\in \R}Q_p(\lambda),\qquad M_p^{nod}(\Omega)=\bigcup_{\lambda\in\R}Q_p^{nod}(\lambda),
\]
where $Q_p(\lambda), Q_p^{nod}(\lambda)$ are the sets in Proposition \ref{prop_J} and in Proposition \ref{deraction}, respectively. Furthermore, for every $p\in(2,2^*)$, we denote
\begin{align}
\label{mup}
\mu_p :=&\, 2\sup\left\{\JJ_\Omega'(\lambda) \mid \JJ_\Omega \text{ is differentiable at } \lambda \right\}, \\
\mu_p^{nod} :=&\, 2\sup\left\{(\JJ_\Omega^{nod})'(\lambda) \mid \JJ_\Omega^{nod} \text{ is differentiable at } \lambda \right\},\label{mupnod}
\end{align}
and we recall that $\mu_N$ is the number defined in \eqref{eq:muN}.

\begin{remark}
\label{darboux}
We will see that although $\JJ_\Omega$ (respectively  $\JJ_\Omega^{nod}$) may fail to be differentiable on a set of measure zero, {\em every} value $\mu \in (0, \frac12 \mu_p)$ (resp. $(0, \frac12 \mu_p^{nod}$)) is achieved by $\JJ_\Omega'$ (resp. $(\JJ_\Omega^{nod})'$). This is the Darboux-type property we mentioned in the Introduction.
\end{remark}

The rest of this section is devoted to show that $\mu_p,\mu_p^{nod}$ above provide the desired thresholds of Theorem \ref{thm:masses}, and that they are equal to $+\infty$ when $p<2+\frac4N$ and finite when $p\geq2+\frac4N$. Since the argument is exactly the same both for action ground states and for nodal action ground states, here we report it in details only for the nodal case. The argument is divided in the following series of lemmas.

\begin{lemma}
\label{minimum}
For every $p\in (2,2^*)$, let
\begin{equation}
\label{defs}
g(\lambda) := \liminf_{\eps \to 0^-}\frac{\JJn(\lambda+\eps)-\JJn(\lambda)}{\eps}\qquad\text{and} \qquad \bargi := 2\sup_{\lambda\in \R}g(\lambda).
\end{equation}
Then the following holds.
\begin{itemize}
\item[1)] if $\mu \in (0,\bargi)$, then $\mu \in M_p^{nod}(\Omega)$;
\item[2)] if $\mu > \bargi$, then $\mu \notin M_p^{nod}(\Omega)$.
\end{itemize}
\end{lemma}

\begin{proof} Note first that, by Proposition \ref{Prop Nodal} and Proposition \ref{deraction}, $\overline{g}>0$. We fix any $\mu \in  (0,\bargi)$ and we take $\barlam$ such that $\mu < 2g(\barlam) \le \bargi$. Note that if $\bargi$ is not attained, then $\barlam$ exists by definition; if $\bargi$ is attained, we take $\barlam$ such that $2g(\barlam) = \bargi$. Since $\mu < 2g(\barlam)$, there exists $\delta >0$ such that for every $\eps \in (-\delta,0)$,
\[
2\frac{\JJn(\barlam+\eps)-\JJn(\barlam)}{\eps} > \frac12(\mu + 2g(\barlam)),
\]
or
\begin{equation}
\label{less}
\JJn(\barlam+\eps) < \JJn(\barlam) + \frac{\eps}4 (\mu + 2g(\barlam)).
\end{equation}
Define now the function $f_\mu : [-\lambda_2, \barlam] \to \R$ as
\[
f_\mu (\lambda) = \JJn(\lambda) - \frac\mu2\lambda.
\]
Since $f_\mu$ is continuous, it has a global minimum point $\widetilde\lambda \in [-\lambda_2, \barlam]$. Notice that $\widetilde\lambda \ne -\lambda_2$ because, by \eqref{up}, $f_\mu(\lambda) < f_\mu(-\lambda_2)$ in a right neighborhood of $-\lambda_2$. Similarly, $\widetilde\lambda \ne \barlam$ because for $\eps \in (-\delta, 0)$, by \eqref{less},
\begin{align*}
f_\mu(\widetilde\lambda) \le f_\mu (\barlam +\eps ) &= \JJn(\barlam+\eps) -\frac\mu2\barlam -\frac\eps 2\mu < \JJn(\barlam) + \frac{\eps}4 (\mu + 2g(\barlam))  -\frac\mu  2\barlam -\frac\eps 2\mu \\
&= f_\mu(\barlam) + \eps\,\frac{2g(\barlam)-\mu}{4} < f_\mu(\barlam).
\end{align*}
Therefore $\widetilde\lambda \in (-\lambda_2, \barlam)$ and $f_\mu(\widetilde\lambda + \eps) - f_\mu(\widetilde\lambda) \ge 0$
for every $|\eps|$ small enough. Now, if $\eps <0$,
\[
\begin{split}
\frac{\JJn(\widetilde\lambda + \eps)-\JJn(\widetilde\lambda )}{\eps}  - \frac\mu2=&\,
\frac{\JJn(\widetilde\lambda + \eps) -\mu\widetilde\lambda/2 -\eps\mu/2 -\JJn(\widetilde\lambda )+\mu\widetilde\lambda/2 }{\eps}\\
=&\,\frac{f_\mu(\widetilde\lambda+\eps)-f_\mu(\widetilde\lambda)}{\eps}\le 0,
\end{split}
\]
whence
\[
2\limsup_{\eps \to 0^-} \frac{\JJn(\widetilde\lambda + \eps)-\JJn(\widetilde\lambda )}{\eps} \le \mu.
\]
The same computation, for $\eps >0$ leads to
\[
2\liminf_{\eps \to 0^+} \frac{\JJn(\widetilde\lambda + \eps)-\JJn(\widetilde\lambda )}{\eps} \ge \mu.
\]
These inequalities, coupled with \eqref{der} show that $\JJn$ is differentiable at $\widetilde\lambda$ and that $(\JJn)'(\widetilde\lambda) = \mu$.
Thus, every  nodal action ground state in $\NN_{\widetilde\lambda}^{nod}(\Omega)$ has mass $\mu$, and $\mu \in M_p^{nod}(\Omega)$. This proves 1).

To prove the second statement, just notice that if $\mu > \bargi$ is the mass of a nodal action ground state $u$ in some $\NN_\lambda^{nod}(\Omega)$, then by \eqref{der},
\[
\bargi < \mu \le \sup Q_p^{nod}(\lambda) \le 2\liminf_{\eps \to 0^-}  \frac{\JJn(\lambda + \eps)-\JJn(\lambda )}{\eps} = 2g(\lambda) \le \bargi,
\]
which is impossible.
\end{proof}

\begin{remark} 
\label{rem:GlobMin}	
The main argument in the proof of the preceding lemma consists in minimizing, given $\mu$, the function $f_\mu(\lambda) = \JJn(\lambda) - \frac\mu 2\lambda$ over the compact set $[-\lambda_2,\barlam]$. This approach is used to unify the three cases of $p$ $L^2$-subcritical, $L^2$-critical or $L^2$-supercritical. We note, for future reference, that for $L^2$-subcritical $p$ and all $\mu$, or for $L^2$-critical $p$ and $\mu < 2\mu_N$, the function $f_\mu$ can be minimized on $[-\lambda_2, +\infty)$, obtaining then a global minimum. Indeed, in these cases $f_\mu$ is coercive, as one immediately sees by writing
\[
f_\mu(\lambda) = \lambda\left( \frac{\JJn(\lambda)}{\lambda}-\frac\mu 2\right)
\]
and taking into account the asymptotic behavior of $\JJn(\lambda)/\lambda$ described in \eqref{J/lambda}. Actually one can easily see that  $f_\mu$ can be minimized on $\R$ (as $\JJn(\lambda) =0$ for $\lambda \le -\lambda_2$).
\end{remark}

\begin{lemma} 
\label{sups}
For every $p\in (2,2^*)$, let $g(\lambda)$ and $\bargi$ be as in \eqref{defs}. Then, recalling \eqref{mupnod},
\[
\mu_p^{nod} = \bargi.
\]
\end{lemma}

\begin{proof} We provide a proof for completeness.
Let $A \subseteq \R$ be the set of points where $\JJn$ is differentiable, so that  $\displaystyle \mu_p^{nod} = 2\sup_A  (\JJn)'$. Obviously, for every  $\lambda \in A$, $(\JJn)'(\lambda) = g(\lambda)$, and therefore
\[
\mu_p^{nod}= 2\sup_{\lambda \in A} (\JJn)'(\lambda) = 2\sup_{\lambda \in A} g(\lambda) \le \bargi.
\]
Conversely, let $\lambda \in \R$ and observe that, for every $\eps <0$, since $\JJn$ is locally Lipschitz continuous,
\[
\JJn(\lambda) - \JJn(\lambda +\eps) = \int_{\lambda+\eps}^\lambda (\JJn)'(t)\,dt \le (-\eps) \sup_{A\cap [\lambda+\eps,\lambda]}  (\JJn)'
\le  (-\eps) \sup_A   (\JJn)' =-\eps \frac{\mu_p^{nod}}2.
\]
Dividing by $-\eps >0$ we obtain
\[
\frac{\JJn(\lambda +\eps) - \JJn(\lambda) }{\eps} \le \frac{\mu_p^{nod}}2,
\]
whence
\[
2g(\lambda) = 2\liminf_{\eps \to 0^-} \frac{\JJn(\lambda +\eps) - \JJn(\lambda) }{\eps} \le \mu_p^{nod}.
\]
Since this holds for every $\lambda \in \R$, we obtain $\bargi \le \mu_p^{nod}$.
\end{proof}

\begin{remark}
\label{newform}
In view of the previous lemma, the conclusion of Lemma \ref{minimum} can be written as
\begin{itemize}
\item[1)] if $\mu \in (0, \mu_p^{nod})$, then $\mu \in M_p^{nod}(\Omega)$;
\item[2)] if $\mu > \mu_p^{nod}$, then $\mu \notin M_p^{nod}(\Omega)$.
\end{itemize}
\end{remark}

\begin{lemma}
\label{J'bounds}
There results
\[
\mu_p^{nod} \begin{cases} =+\infty & \text{ if } p < 2 +4/N \\ < +\infty & \text{ if } p \ge 2 +4/N. \end{cases}
\]
\end{lemma}

\begin{proof}
First let $p < 2 +4/N$. By Remark \ref{newform}, it is enough to show that there exist nodal action ground states of arbitrarily large mass. We will do it by estimating from below the mass of elements of $\NN_\lambda^{nod}(\Omega)$ in terms of $\lambda$. To this aim, take any $u\in \NN_\lambda^{nod}(\Omega)$. By the Gagliardo-Nirenberg inequality
\[
\|u\|_p^p\leq K_p\|u\|_{2}^{p-\alpha}\|\nabla u\|_{2}^\alpha, \qquad \alpha = N\left(\frac{p}2 -1\right),
\]
noticing that $\alpha <2$ since $p$ is $L^2$-subcritical, using the Young inequality and writing $\mu = \|u\|_2^2$, we deduce that
\[
\|\nabla u \|_2^2 +\lambda\mu= \|u\|_p^p\leq K_p  \frac{2-\alpha}2 \mu^{\frac{p-\alpha}{2}\frac{2}{2-\alpha}} + \frac\alpha{2}\|\nabla u\|_{2}^2 = 
 K_p  \frac{2-\alpha}2 \mu^{\frac{p-\alpha}{2-\alpha}} +\frac\alpha{2}\|\nabla u\|_{2}^2.
\]
Therefore
\[
\lambda\mu \le \left(1 - \frac\alpha{2}\right)\|\nabla u \|_2^2 + \lambda \mu\le  K_p  \frac{2-\alpha}2 \mu^{\frac{p-\alpha}{2-\alpha}}, 
\] 
whence $\mu \ge C \lambda^{\frac{2-\alpha}{p-2}} \to +\infty$ as $\lambda \to +\infty$.

Assume now that $p \ge 2 + 4/N$ and take again $u \in \NN_\lambda^{nod}(\Omega)$. Denoting by $C$ (possibly different but) universal constants, by \eqref{up} we notice that 
\[
\mu = \|u\|_2^2 \le C\|u\|_p^2 = C \left(\JJn(\lambda)\right)^\frac2p \le C(\lambda + \lambda_2)^{\frac2{p-2}}.
\]
This shows that $\mu$ is bounded when  $\lambda$ ranges in a bounded subset of $\R$. By \eqref{J/lambda}, $\JJn(\lambda)/\lambda$ is bounded on $[1,+\infty)$ since $p\ge 2 +4/N$. Therefore, for every $\lambda \ge 1$,
\[
\lambda\mu = \lambda\|u\|_2^2 \le \|\nabla u \|_2^2 +\lambda \|u\|_2^2= \|u\|_p^p = \frac1\kappa \JJn(\lambda) \le C\lambda
\]
and the proof is complete.
\end{proof}

\begin{remark} We notice that, when $p = 2+4/N$, recalling \eqref{eq:muN},
\label{SmuN}
\[
\mu_p^{nod}\ge 2\mu_N.
\]
Indeed, for every $\lambda >-\lambda_2$, since $\JJn$ is locally Lipschitz continuous,
\[
\JJn(\lambda) = \int_{-\lambda_2}^\lambda (\JJn)'(t)\,dt \le (\lambda+\lambda_2)\frac{\mu_p^{nod}}2
\]
so that, by \eqref{J/lambda},
\[
\mu_N = \lim_{\lambda \to \infty} \frac{\JJn(\lambda)}{\lambda} \le \lim_{\lambda \to \infty} \frac{\lambda+\lambda_2}{\lambda}\frac{\mu_p^{nod}}2 = \frac{\mu_p^{nod}}2. 
\]
\end{remark}
\bigskip

\begin{proof}[Proof of Theorem \ref{thm:masses}, nodal case] Point $(i)$ follows directly by Lemmas \ref{minimum}--\ref{sups}--\ref{J'bounds}. When $p \ge 2+4/N$, the same results show that $\mu_p^{nod}$ is finite, $(0,\mu_p^{nod})\subset M_p^{nod}(\Omega)$, and that $M_p^{nod}(\Omega) \cap (\mu_p^{nod},+\infty) = \emptyset$. This proves point $(ii)$ for $p=2+\frac4N$ and, to complete the proof of point $(iii)$ too, we are left to show that $\mu_p^{nod} \in M_p^{nod}(\Omega)$ when $p$ is $L^2$-supercritical. To this aim, let $\mu_n \nearrow \mu_p^{nod}$ and let $u_n$ be a sequence of nodal action ground states of mass $\mu_n$ in some $\NN_{\lambda_n}^{nod}(\Omega)$, with $\lambda_n \in (-\lambda_2,+\infty)$ for every $n$. We notice that no subsequence of $\lambda_n$ can converge to $-\lambda_2$, since in this case, for such subsequence (not relabeled), we would have
\[
\mu_n = \|u_n\|_2^2 \le C \|u_n\|_p^2 = C\left(\frac1\kappa \JJn(\lambda_n)\right)^{2/p} \to 0
\]
by the continuity of $ \JJn$. Similarly, no subsequence of $\lambda_n$ can tend to $+\infty$, because we would have
\[
\lambda_n \mu_n \le \|\nabla u_n\|_2^2 + \lambda_n\mu_n = \|u_n\|_p^p = \frac1\kappa \JJn(\lambda_n),
\]
entailing that $\mu_n \le  \frac1\kappa \JJn(\lambda_n)/\lambda_n \to 0$ as $n \to \infty$ by Proposition \ref{lem:J} (since $p>2+\frac4N$). Therefore, we can assume that, up to subsequences, $\lambda_n \to \lambda \in (-\lambda_2,+\infty)$ as $n\to \infty$. Then, by Proposition \ref{convergence_NGS}, up to subsequences we have that $u_n$ converges in $H^1_0(\Omega)$ to a nodal action ground state $u \in \NN_{\lambda}^{nod}(\Omega)$. Since $\|u\|_2^2 = \mu_p^{nod}$, we see that $\mu_p^{nod} \in M_p^{nod}(\Omega)$ and the proof is complete. 
\end{proof}

\section{Proof of Theorems \ref{thm:exnod}--\ref{thm:LENS}--\ref{thm:super}}
\label{sec:super}

In this section we prove the main results of the paper with respect to normalized nodal solutions and least energy normalized solutions, i.e. Theorems \ref{thm:exnod}--\ref{thm:LENS}--\ref{thm:super}.

The proof of Theorems \ref{thm:exnod}--\ref{thm:LENS} is a direct consequence of the discussion developed in Section \ref{sec:masses}. In particular, Theorem \ref{thm:exnod} is a corollary of Theorem \ref{thm:masses}.  

\begin{proof}[Proof of Theorems \ref{thm:LENS}]
	Fix any $\mu >0$ if $p< 2+4/N$ or $\mu \in (0,2\mu_N)$ if $p=2+4/N$ (where $\mu_N$ is as in \eqref{eq:muN}), and let 
	$\wlambda$ be a global minimizer of the function $f_\mu(\lambda) = \JJn(\lambda) -\frac\mu 2\lambda$ (see Remark \ref{rem:GlobMin}). Take 
	$\wu$ to be a nodal action ground state of mass $\mu$ corresponding to $\wlambda$
	 (when $p=2+\frac4N$, the fact that such a $\wu$ exists for every $\mu<2\mu_N$ is guaranteed by Remark \ref{SmuN}).  Let then $w$ be any other nodal solution of \eqref{nlse}, for some $\lambda_w \in \R$. Then
	\begin{align*}
	E(w,\Omega) &= J_{\lambda_w}(w,\Omega) -\frac\mu 2\lambda_w \ge \JJn(\lambda_w) -\frac\mu 2\lambda_w \ge \min_{\lambda \in \R} \left(\JJn(\lambda)-\frac\mu 2\lambda \right) \\
	& = \JJn(\wlambda)-\frac\mu 2\wlambda
	= J_\wlambda(\wu,\Omega) -\frac\mu 2\wlambda = E(\wu,\Omega),
	\end{align*}
	that is the nodal action ground state $\wu$ is a least energy normalized nodal solution with mass $\mu$. Moreover, if $w$ is another least energy normalized nodal solution with mass $\mu$, the previous lines become a chain of equalities, showing in particular that $J_{\lambda_w}(w,\Omega)=\JJ_\Omega^{nod}(\lambda_w)$, i.e. $w$ is a nodal action ground state in $\NN_{\lambda_w}^{nod}(\Omega)$.
\end{proof}

We are thus left to prove Theorem \ref {thm:super}. To this end, in what follows we take
\[
p>p_c:=2+\frac4N,\qquad\Omega\text{ star-shaped with respect to }0.
\]
As is quite common in the literature, working on star-shaped domains allows one to profit of the Poho\v{z}aev identity. In particular, we will use the following fact.

\begin{proposition}
\label{lower_bound_energy}
Assume that $\Omega\subset\R^N$ is bounded, smooth and star-shaped. 
Then, for every solution $u$ of \eqref{nlseNOMASS},     
\begin{equation*}
E(u,\Omega) \ge \frac{N(p - p_c)}{4p} \| u \|_p^p.
\end{equation*}
\end{proposition}

\begin{proof}
Up to translating the domain, we may assume that $\Omega$ is star-shaped with respect to $0$. As $\Omega$ is regular, an $H^1$ solution $u$ of \eqref{nlseNOMASS} is in ${\mathcal C}^2(\overline\Omega)$ and the Poho\v{z}aev identity (see e.g. \cite[Chapter III, Lemma 1.4]{Str}) implies that
    \begin{equation*}
        \frac{N-2}{2} \| \nabla u \|_2^2
        - \frac{N}{p} \| u \|_p^p
        + \frac{\lambda N}{2} \| u \|_2^2
        + \frac{1}{2} \int_{\partial \Omega} | \partial_{\nu}u |^2 x \cdot \nu\, d \sigma = 0,
    \end{equation*}
    where $\nu$ is the exterior unit normal.
    Since $\Omega$ is star-shaped,  $x \cdot \nu \ge 0$ for every $x \in \partial \Omega$,
    so that 
    \begin{equation*}
        \frac{N-2}{2} \| \nabla u \|_2^2
        - \frac{N}{p} \| u \|_p^p
        + \frac{\lambda N}{2} \| u \|_2^2
        \le 0.
    \end{equation*}
    Recalling that $\| \nabla u \|_2^2 + \lambda \| u \|_2^2 = \| u \|_p^p$ because $u\in\NN_{\lambda}(\Omega)$, we see that
    \begin{equation*}
        \Bigl( \frac{N-2}{2} - \frac{N}{p} \Bigr) \| u \|_p^p
        + \lambda \| u \|_2^2
        \le 0,
    \end{equation*}
entailing
    \begin{equation*}
        E(u,\Omega)
        = J_{\lambda}(u,\Omega) - \frac{\lambda \| u \|_2^2}{2}
        \ge \Bigl( \frac{p-2}{2p} + \frac{p(N-2) - 2N}{4p} \Bigr) \|u\|_p^p
        = \frac{N(p-p_c)}{4p} \| u \|_p^p. \qedhere
    \end{equation*}
\end{proof}

\begin{proof}[Proof of Theorem \ref{thm:super}] The argument is divided in two steps.

\smallskip
{\em Step 1.}  Letting $\mu_p$ the threshold given by Theorem \ref{thm:exnod}$(iii)$, we prove that there exists a least energy normalized solution for every $\mu\leq\mu_p$. To this end, let $\mu\leq\mu_p$ be fixed and
\[
\Sf_{\mu}:=\left\{u\in H_0^1(\Omega)\,:\,u \text{ solves \eqref{nlse} for some }\lambda\in\R\right\}
\]
be the set of all solutions of \eqref{nlse}. Since $\Sf_{\mu}\neq\emptyset$ by Theorem \ref{thm:exnod}$(iii)$, let $(u_n)_{n} \subset\Sf_\mu$ be such that
    \begin{equation*}
        -\Delta u_n + \lambda_n u_n = |u_n|^{p-2}u_n\qquad\text{and}\qquad
        E(u_n,\Omega) \xrightarrow[n \rightarrow \infty]{} \inf_{v \in \Sf_{\mu}} E(v,\Omega).
    \end{equation*}
    Observe first that, since for all $n$ we have
    \begin{equation*}
        E(u_n,\Omega)
        = J_{\lambda_n}(u_n,\Omega) - \frac{\lambda_n}2 \mu
        \ge - \frac{\lambda_n}2 \mu,
    \end{equation*}
    the sequence $(\lambda_n)_n$ is bounded from below. Furthermore,
    according to Proposition~\ref{lower_bound_energy}, as $u_n\in \NN_{\lambda_n}(\Omega)$, we have
    \begin{equation*}
        E(u_n,\Omega)
        \ge \frac{N(p-p_c)}{2(p-2)} J_{\lambda_n}(u_n,\Omega)
        \ge \frac{N(p-p_c)}{2(p-2)} \inf_{v \in \NN_{\lambda_n}(\R^N)} J(v,\R^N)
        = \frac{N(p-p_c)}{2(p-2)} \JJ_{\R^N}(\lambda_n).
    \end{equation*}
    Recalling that $\displaystyle \JJ_{\R^N}(s)= s^{\frac{2N-p(N-2)}{2(p-2)}} \JJ_{\R^N}(1) \to +\infty$ as $s \to +\infty$, and since $p_c<p<2^*$,
    this implies that $(\lambda_n)_n$ is also bounded from above and hence, up to subsequences, $\lambda_n \rightarrow \lambda$ for some $\lambda \in \R$.
	Since $(E(u_n,\Omega))_n$ is bounded and $\|u_n\|_2 = \mu$, so is $(J_{\lambda_n}(u_n,\Omega))_n$. 
  As usual, this implies that $(u_n)_n$ is bounded in $H^1(\Omega)$. Therefore, we may assume that $u_n \rightharpoonup u$ in $H^1(\Omega)$ and $u_n \to u$ in $L^q(\Omega)$ for every $q \in [2,2^*)$, showing that $u$ is a solution to \eqref{nlseNOMASS} of mass $\mu$.
  Moreover, by weak lower semi-continuity,
    \begin{equation*}
        E(u,\Omega) \le \liminf_{n\to\infty}E(u_n,\Omega)= \inf_{v \in \Sf_{\mu}} E(v,\Omega),
    \end{equation*}
    so that $u$ is a least energy normalized solution of mass $\mu$.
    
   	Arguing analogously one can prove that, taking $\mu_p^{nod}$ as in Theorem \ref{thm:exnod} and letting
   	\[
   	\mathcal{S}_\mu^{nod}:=\left\{u\in H_0^1(\Omega)\,:\,u\text{ solves \eqref{nlse} for some }\lambda\in\R,\text{ }u^\pm\not\equiv0\right\}
   	\]
    the set of all nodal solutions of \eqref{nlse}, for every $\mu\leq\mu_p^{nod}$ there exists a least energy normalized nodal solution, that is $u\in\Sf_\mu^{nod}$ such that
    \[
    E(u,\Omega)=\inf_{v\in\Sf_\mu^{nod}}E(v,\Omega).
    \]
    Indeed, the argument above here shows again that, up to subsequences,  a sequence $(u_n)_n\subset\Sf_\mu^{nod}$ such that $E(u_n,\Omega)\to\inf_{v\in\Sf_\mu^{nod}}E(v,\Omega)$ converges weakly in $H^1(\Omega)$ to a solution $u$ of \eqref{nlse}, and to conclude that $u$ is a least energy normalized nodal solution it remains to show that $u$ is still nodal. To see this, we notice that, by strong convergence in $L^2(\Omega)$, $u \not\equiv 0$. If it were $u \geq 0$, by Remark~\ref{no_pos_sol} we would have  $\lambda > -\lambda_1$.
    But then, 
    \begin{equation*}
   \kappa\liminf_n \|u_n^\pm\|_p^p =    \liminf_n J_{\lambda_n}(u_n^\pm,\Omega) \ge \liminf_n \mathcal{J}_\Omega(\lambda_n) = \mathcal{J}_\Omega(\lambda) > 0,
    \end{equation*}
    so that neither $(u_n^+)_n$ nor $(u_n^-)_n$ converge to zero in $L^p(\Omega)$,
    contradicting the assumption $u \geq 0$.

\smallskip
{\em Step 2.} We are left to prove that, when the mass is sufficiently small, least energy normalized solutions/least energy normalized nodal solutions are action ground states/nodal action ground states with frequency uniformly bounded from above. We give the details of the proof for least energy normalized nodal solutions, the signed case being analogous.

Set 
\[
\overline{\lambda}_p^{nod}:=\frac{2p\lambda_2}{N(p-p_c)},\qquad\overline{\mu}_p^{nod}:=\frac{2\JJ_\Omega^{nod}(\overline{\lambda}_p^{nod})}{\overline{\lambda}_p^{nod}+\lambda_2}.
\]
Note that, since $\JJ_\Omega^{nod}$ is locally Lipschitz by Proposition \ref{deraction} and $\JJ_\Omega^{nod}(-\lambda_2)=0$ by Proposition \ref{properties_NGS}, then
\[
\JJ_\Omega^{nod}\big(\overline{\lambda}_p^{nod}\big)=\int_{-\lambda_2}^{\overline{\lambda}_p^{nod}}(\JJ_\Omega^{nod})'(s)\,ds\leq (\overline{\lambda}_p^{nod}+\lambda_2)\sup_{\lambda}(\JJ_{\Omega}^{nod})'(\lambda),
\]
so that $\overline{\mu}_p^{nod}\leq\mu_p^{nod}$, where $\mu_p^{nod}$ is the number defined in \eqref{mupnod}.
Moreover, by Proposition \ref{lower_bound_energy} and the definition of $\overline{\lambda}_p^{nod}$, if $u\in\NN_{\lambda}^{nod}(\Omega)$ with $\lambda\geq\overline{\lambda}_p^{nod}$ and $\|u\|_2^2=\mu$, then
   \begin{equation}
\label{energy_lower}
E(u,\Omega)
\ge  \frac{N(p - p_c)}{4p} \| u \|_p^p = \frac{N(p - p_c)}{4p} \left(\|\nabla u\|_2^2+\lambda\| u \|_2^2\right)
> \frac{N(p - p_c)}{4p} \lambda \mu\geq\frac{\lambda_2}2\mu.
\end{equation}
Let then $\mu\leq\overline{\mu}_p^{nod}$ be fixed. Since $\mu\leq\mu_p^{nod}$, the proof of Lemma \ref{minimum} above guarantees that there exists a function $u\in\Sf_\mu^{nod}$ such that $E(u,\Omega)< \JJ_\Omega^{nod}(-\lambda_2)+\frac{\lambda_2}2\mu=\frac{\lambda_2}2\mu$, so that
\[
\inf_{v\in\Sf_\mu^{nod}}E(v,\Omega)<\frac{\lambda_2}2\mu.
\]
By \eqref{energy_lower}, any least energy normalized nodal solution in $\Sf_\mu^{nod}$ (which exist by Step 1) must then belong to $\NN_{\lambda}^{nod}(\Omega)$ for some $\lambda<\overline{\lambda}_p^{nod}$. Since $\overline\lambda_p^{nod}$ depends only on $p$ and $\Omega$, this proves that least energy normalized nodal solutions have uniformly bounded frequency, and 
\[
\inf_{v\in\Sf_\mu^{nod}}E(v,\Omega)=\inf_{\substack{v\in\Sf_\mu^{nod}\cap\NN_{\lambda}^{nod}(\Omega) \\ \lambda<\overline\lambda_p^{nod}}}E(v,\Omega).
\]
To conclude, it remains to show that least energy normalized nodal solutions are always nodal action ground states for $\mu \le \overline\mu_p^{nod}$. This follows arguing exactly as in the proof of Lemma \ref{minimum} and of Theorem \ref{thm:LENS}, noting that the function $f_\mu(\lambda):=\JJ_\Omega^{nod}(\lambda)-\frac\lambda2\mu$ has a minimum point in $(-\lambda_2, \overline\lambda_p^{nod})$, since
it is continuous, $f_\mu(s) < f_\mu(-\lambda_2)$ for $s$ in a right neighbourhood of $-\lambda_2$, and  $f_\mu\big(\overline\lambda_p^{nod}\big)=\JJ_{\Omega}^{nod}\big(\overline\lambda_p^{nod}\big)-\frac{\overline\lambda_p^{nod}}2\mu\geq\frac{\lambda_2}2\mu=f_\mu(-\lambda_2)$.
\end{proof}

\section*{Acknowledgements}
D.G. is an F.R.S.-FNRS Research Fellow.
All authors acknowledge that this work has been partially supported by the IEA Project ``Nonlinear Schr\"odinger Equations on Metric Graphs''. 
C.D. and D.G. acknowledge that this work has been carried out in the framework of the project NQG (ANR-23-CE40-0005-01) funded by the French National Research Agency (ANR). 

S.D. acknowledges that this study was carried out within the project E53D23005450006 ``Nonlinear dispersive equations in presence of singularities'' -- funded by European Union -- Next Generation EU within the PRIN 2022 program (D.D. 104 -02/02/2022 Ministero dell'Universit\'a e della Ricerca). This manuscript reflects only the author's views and opinions and the Ministry cannot be considered responsible for them.


\end{document}